\documentclass[11pt]{amsart}
\usepackage{cases}

\theoremstyle{plain}
\newtheorem{theorem}                {Theorem}      [section]
\newtheorem*{theorem*}                {Theorem}
\newtheorem{proposition}  [theorem]  {Proposition}

\newtheorem{lemma}        [theorem]  {Lemma}

\theoremstyle{definition}

\newtheorem{remark}       [theorem]  {Remark}
\newtheorem{definition}   [theorem]  {Definition}

\DeclareMathOperator{\trace}{trace} 

 \DeclareMathOperator{\id}{I}

\DeclareMathOperator{\cst}{constant}

\DeclareMathOperator{\grad}{grad}

\DeclareMathOperator{\ric}{Ric}
\DeclareMathOperator{\nil}{Nil_3}

\numberwithin{equation}{section}

\begin{document}

\title[Simons Formulas]{Simons formulas in complex space forms and product spaces}

\author{Dorel~Fetcu}
\author{Stefano~Nardulli}

\thanks{This work was partially supported by the Grant No. $2026/02612-3$, from FAPESP, Brazil. S. Nardulli was supported by FAPESP Auxilio Jovem Pesquisador No. $2021/05256-0$, CNPq Bolsa de Produtividade em Pesquisa 1D, No. $12327/2021-8$, $23/08246-0$, Geometric Variational Problems in Smooth and Nonsmooth
Metric Spaces No. $441922/2023-6$.}

\address{Department of Mathematics and Informatics\\
Gh. Asachi Technical University of Iasi\\
Bd. Carol I, 11A \\
700506 Iasi, Romania} \email{dorel.fetcu@academic.tuiasi.ro}

\address{Department of Mathematics\\ Federal University of ABC\\
09210-580, Santo Andr\'e-SP, Brazil}
\email{stefano.nardulli@ufabc.edu.br}



\begin{abstract} We compute Simons type equations for parallel mean curvature submanifolds in complex space forms $N^n(c)$, with constant holomorphic sectional curvature $c$, and product spaces $N^n(c)\times\mathbb{R}$. These formulas are then used to characterize some of these submanifolds. 
\end{abstract}

\maketitle

\section{Introduction}

One of the most influential article in modern Differential Geometry, was published almost six decades ago, in 1968. In this paper, J. Simons obtained an equation for the Laplacian of the second fundamental form of a minimal subamanifold of a Riemannian manifold (see \cite{JS}). This formula was then used to characterize certain minimal submanifolds of spheres. Very fast, such formulas, that one calls nowadays Simons type equations, were also developed for CMC hypersurfaces in space forms (see \cite{NS}) and then for constant mean curvature (CMC) and parallel mean curvature (PMC) submanifolds in space forms (see, e.g., \cite{AdC,AT,WS,BS,Y1}). Almost ten years later, in~\cite{CY}, it was proved a general formula of this type for a symmetric $(1,1)$-tensor defined on a Riemannian manifold. All these equations use the fact that the shape operator $A$ of a submanifold in a space form satisfies the classical Codazzi property $(\nabla_XA)Y=(\nabla_YA)X$.

However, in other ambient spaces, the shape operator $A$ may fail to satisfy this property and the situation becomes more complicated, as shown, for example, in \cite{B}. In the case of CMC surfaces in most of the homogeneous $3$-manifolds this problem was solved by using another operator, obtained from the Abresch-Rosenberg differential. This differential, introduced by U.~Abresch and H.~Rosenberg in \cite{AR,AR2} , is the traceless part of a certain quadratic form defined on surfaces of $\mathbb{S}^2\times\mathbb{R}$, $\mathbb{H}^2\times\mathbb{R}$, $\widetilde{PSL}_2(\mathbb{R})$, and $\nil$. The differential is holomorphic if and only if the surface is CMC. In order to exploit this property, in \cite{B}, one has defined an operator $S$, related to the Abresch-Rosenberg differential, on a CMC surface of a product space, which satisfies the classical Codazzi property. Then a Simons type equation for $S$ rather than the shape operator $A$ was obtained. This result was then generalized to all other spaces where Abresch-Rosenberg differentials do exist (see \cite{ET}). Similar results were obtained for PMC surfaces in complex space forms (see \cite{FP}). 

One section of our paper is devoted to the study of PMC surfaces in a product space $N^n(c)\times\mathbb{R}$, where $N^n(c)$ is a complex space form with constant holomorphic sectional curvature $c$, by means of developing and using two Simons type equations for operators provided by the holomorphic differentials introduced in \cite{FR-TAMS}.

When dealing with higher dimensional submanifolds in ambient spaces where the shape operator does not satisfy the classical Codazzi property, one does not have the alternative offered by the use of holomorphic differentials defined on CMC or PMC surfaces, but, sometimes, for very special submanifolds, one could still obtain Simons type equations involving the shape operator. Such a formula was proved, for example, for Lagrangian submanifolds in a complex space form (see \cite{Chen77}) (and also for totally real surfaces in a complex space form (see \cite{Chen74})). 

In the second part of our paper, we consider PMC submanifolds in a complex space form $N^n(c)$ or in a product space $N^n(c)\times\mathbb{R}$ and compute the Laplacian of the squared norm of $A_H$, where $A$ is the shape operator and $H$ is the mean curvature vector field. Then we use these two Simons type formulas to characterize minimal and CMC hypersurfaces, PMC totally real and PMC anti-invariant submanifolds, under additional geometric hypotheses. The main results here are describing gap phenomena, providing information about the admissible range for the norm of the shape operator or for the mean curvature. This is to be expected when one uses inequalities derived from Simons type formulas, although such equations have a wider range of application in the theory of submanifolds. For other ways of working with geometric inequalities we refer, for example, to the very recent paper \cite{NG}.

\noindent\textbf{Acknowledgments.} The first author would like to thank the Department of Mathematics of the Federal University of ABC in Santo Andr\'e, SP, Brazil, for providing a very stimulative work environment during the preparation of this paper.

\noindent{\bf Conventions.} We use the following definitions and sign conventions. Consider $\phi:M\rightarrow N$ be a submanifold in a Riemannian manifold. In general, we will not indicate explicitly the Riemannian metrics on $M$ or $N$. Then 
$$
\Delta=\trace(\nabla^{\phi})^2 =\trace(\nabla^{\phi}\nabla^{\phi}-\nabla^{\phi}_{\nabla})
$$ 
is the rough Laplacian defined on the set of all sections in $\phi^{-1}(TN)$, that is on $C(\phi^{-1}(TN))$, and $\bar R$ is the curvature tensor field of $N$, given by 
$$
\bar R(X,Y)Z=[\bar\nabla_X,\bar\nabla_Y]Z-\bar\nabla_{[X,Y]}Z.
$$
Here, $\nabla^{\phi}$ denotes the pull-back connection on $\phi^{-1}(TN)$, while $\nabla$ and $\bar\nabla$ are the Levi-Civita connections on $TM$ and $TN$, respectively.

\section{Preliminaries}

Let $N^n(c)$ be a complex space form with complex dimension $n$, complex
structure $(J,\langle,\rangle)$, and constant holomorphic
sectional curvature $c$, i.e., $N^n(c)$ is either $\mathbb{C}P^n(c)$,
or $\mathbb{C}^n$, or $\mathbb{C}H^n(c)$, as $c>0$,
$c=0$, and $c<0$, respectively. Then the curvature tensor of $N(c)$ is given by
\begin{eqnarray}\label{eq:curv_cpn}
\bar R(X,Y)Z&=&\frac{c}{4}\{\langle Y,Z\rangle X-\langle X,Z\rangle Y+\langle JY,Z\rangle JX-\langle JX,Z\rangle JY\\\nonumber &&+2\langle
X,JY\rangle JZ\}.
\end{eqnarray}

A submanifold of a complex space form with the property that $J$ maps the tangent bundle into the normal one is called a \textit{totally real} submanifold.

Consider now the product space $N^n(c)\times\mathbb{R}$ and define the following tensors:
$$
\varphi:=J\circ d\pi,\quad\xi:=\frac{\partial}{\partial t},\quad\eta:=dt,\quad\textnormal{and}\quad
\langle,\rangle:=\langle,\rangle_{N^n(c)}+dt\otimes dt,
$$
where $\pi:N^n(c)\times\mathbb{R}\rightarrow N^n(c)$ is the projection map and $t$ is the standard coordinate function on the real axis. Then these tensors define a cosymplectic structure with constant $\varphi$-sectional curvature equal to $c$, as we shall explain in the following.

An \textit{almost contact metric structure} on an odd-dimensional manifold
$\tilde N^{2n+1}$ is given by $(\varphi,\xi,\eta,\langle,\rangle)$, where $\varphi$ is
a tensor field of type $(1,1)$ on $\tilde N$, $\xi$ is a vector field,
$\eta$ is its dual $1$-form with respect to $\langle,\rangle$, while $\langle,\rangle$ is a Riemannian metric such that
$$
\varphi^{2}X=-X+\langle X,\xi\rangle\xi\quad\textnormal{and}\quad
\langle\varphi X,\varphi Y\rangle=\langle X,Y\rangle-\eta(X)\eta(Y),
$$
for any tangent vector fields $X$ and $Y$. The fundamental $2$-form is defined as $\Omega(X,Y)=\langle X,\varphi Y\rangle$. Such a structure is
called {\it normal} if
$$
\mathcal N_{\varphi}(X,Y)+2d\eta(X,Y)\xi=0,
$$
where $\mathcal N_{\varphi}$ is the Nijenhuis tensor.

An almost contact metric manifold $(N,\varphi,\xi,\eta,\langle,\rangle)$ is a \textit{cosymplectic manifold} if it is normal with $\eta$ and $\Omega$ closed.

Equivalently, an almost contact metric manifold is cosymplectic if and only if $\varphi$ is parallel, i.e., $\tilde\nabla\varphi=0$, where $\tilde\nabla$ is the Levi-Civita connection. This implies that also $\xi$ and $\eta$ are parallel. A cosymplectic manifold has a natural local product structure as a product of a K\"ahler manifold and a $1$-dimensional manifold but there exist compact cosymplectic manifolds which are not global products (for more details on cosymplectic manifolds, see, for example, \cite{ABC,B,CdLM}). 

Next, let $\tilde N$ be a cosymplectic manifold. The
sectional curvature of a $2$-plane generated by $X$ and $\varphi X$,
where $X$ is a unit vector orthogonal to $\xi$, is called the
\textit{$\varphi$-sectional curvature} determined by $X$. A cosymplectic
manifold with constant $\varphi$-sectional curvature $c$ is called a
\textit{cosymplectic space form} and will be denoted by $\tilde N(c)$. The curvature
tensor field of $\tilde N(c)$ is given by
\begin{eqnarray}\label{eq:curv_product}
\tilde R(X,Y)Z&=&\frac{c}{4}\{\langle Y,Z\rangle X-\langle X,Z\rangle Y+\langle X,\varphi Z\rangle \varphi Y-\langle Y,\varphi Z\rangle\varphi X\\
\nonumber &&+2\langle X,\varphi Y\rangle\varphi Z+\eta(X)\eta(Z)Y-\eta(Y)\eta(Z)X\\\nonumber &&+\langle X,Z\rangle\eta(Y)\xi-\langle Y,Z\rangle\eta(X)\xi\}.
\end{eqnarray}

A submanifold of a cosymplectic manifold is called \textit{anti-invariant} if $\varphi$ maps the tangent bundle into the normal one.

Now, let $\Sigma^m$, $m\leq 2n-1$, be an $m$-dimensional submanifold of a complex space form $N^n(c)$. For any vector field $X$ tangent to $\Sigma$ we will write $JX=tX+nX$, where $tX$ stands for the tangent part of $JX$ and $nX$ for the normal one. From the equation of Gauss
\begin{eqnarray*}
\langle R(X,Y)Z,W\rangle&=&\langle\bar R(X,Y)Z,W\rangle\\ &&+\sum_{\alpha=m+1}^{2n}\{\langle A_{\alpha}Y,Z\rangle\langle A_{\alpha}X,W\rangle-\langle A_{\alpha}X,Z\rangle\langle A_{\alpha}Y,W\rangle\},
\end{eqnarray*}
and \eqref{eq:curv_cpn} one obtains the expression of the curvature tensor
\begin{eqnarray}\label{eq:R_cpn}
R(X,Y)Z&=&\frac{c}{4}\{\langle Y,Z\rangle X-\langle X,Z\rangle Y+\langle JY,Z\rangle tX-\langle JX,Z\rangle tY\\\nonumber &&+2\langle
X,JY\rangle tZ\}+\sum_{\alpha=m+1}^{2n}\{\langle A_{\alpha}Y,Z\rangle A_{\alpha}X-\langle A_{\alpha}X,Z\rangle A_{\alpha}Y\},
\end{eqnarray}
for all vector fields $X$, $Y$, $Z$, and $W$ tangent to $\Sigma^m$. The shape operator $A$ is given by the equation of Weingarten
$$
\bar\nabla_XV=-A_VX+\nabla^{\perp}_XV,
$$
for any normal vector field $V$. Here $\nabla^{\perp}$ denotes the connection in the normal bundle, and
$A_{\alpha}=A_{E_{\alpha}}$, $\{E_{\alpha}\}_{\alpha=m+1}^{n+1}$
is a local orthonormal frame field in the normal bundle. The connection $\nabla$ on $\Sigma$ and the second fundamental form $\sigma$ are given by the Gauss equation
$$
\bar\nabla_XY=\nabla_XY+\sigma(X,Y).
$$

If $\Sigma^m$, $m\leq 2n$ is a submanifold in a product space $N^n(c)\times\mathbb{R}$, then, from \eqref{eq:curv_product}, we compute its curvature tensor as
\begin{eqnarray}\label{eq:R_product}
R(X,Y)Z&=&\frac{c}{4}\{\langle Y,Z\rangle X-\langle X,Z\rangle Y+\langle X,\varphi Z\rangle tY-\langle Y,\varphi Z\rangle tX\\
\nonumber &&+2\langle X,\varphi Y\rangle tZ+\eta(X)\eta(Z)Y-\eta(Y)\eta(Z)X\\\nonumber &&+\langle X,Z\rangle\eta(Y)T-\langle Y,Z\rangle\eta(X)T\}\\\nonumber &&+\sum_{\alpha=m+1}^{2n+1}\{\langle A_{\alpha}Y,Z\rangle A_{\alpha}X-\langle A_{\alpha}X,Z\rangle A_{\alpha}Y\},
\end{eqnarray}
where we decomposed the vector fields $\varphi X=tX+nX$ and $\xi=T+N$ in their tangent and normal parts, respectively.

\begin{definition} If the mean curvature vector field $H$ of a submanifold $\Sigma^m$ of a Riemannian manifold is
parallel in the normal bundle, i.e., $\nabla^{\perp}H=0$, then $\Sigma^m$ is called a \textit{parallel mean curvature (PMC) submanifold}. If the mean curvature $|H|$ of $\Sigma^m$ is constant, then $\Sigma^m$ is a called a \textit{constant mean curvature (CMC) submanifold}.
\end{definition}

In order to define biharmonic submanifolds, let us first consider a smooth map $\phi: M^{m} \rightarrow \bar M^{n}$ between two Riemannian manifolds. Then $\phi$ is called a \textit{biharmonic map} if it is a critical point of the bienergy functional, assuming $M$ compact,
$$
E_{2}:C^{\infty}(M,\bar M)\to \mathbb{R}, \quad E_{2}(\phi)=\frac{1}{2}\int_{M} \vert \tau(\phi)\vert ^{2} dv,
$$
and $\tau(\phi)= \trace  \nabla d \phi$ is the tension field of
$\phi$. These maps are characterized by the Euler-Lagrange equation, also known as the biharmonic equation (see \cite{Jiang}):
\begin{eqnarray}\label{tau-2}
\tau_{2}(\phi)=\Delta \tau(\phi)-\trace \overline{R}(d\phi(\cdot),\tau(\phi))d \phi (\cdot)=0,
\end{eqnarray}
where $\tau_{2}(\phi)$ is {\it the bitension field} of $\phi$. If $M$ is not compact, then $\phi$ is called biharmonic if it satisfies equation \eqref{tau-2}.

As any harmonic map is biharmonic, we are interested in studying non-harmonic biharmonic  maps, which are called {\it proper-biharmonic maps}.

A \textit{biharmonic submanifold} in a Riemannian manifold is a submanifold for which the inclusion map is biharmonic. In Euclidean space this notion of a biharmonic submanifold coincides with that in \cite{BYC}, as in both cases biharmonic submanifolds are characterized by the equation $\Delta H=0$ (for a detailed account on biharmonic maps and submanifolds, see \cite{Chen-Ou}).

\begin{theorem}[\cite{BMO,Ou}]\label{thm:split} A submanifold $\Sigma^m$ in a Riemannian manifold $\bar M$ is biharmonic if and only if
\begin{equation}\label{eq:bih}
\begin{cases}
-\Delta^{\perp}H+\trace\sigma(\cdot,A_H\cdot)+\trace(\bar R(\cdot,H)\cdot)^{\perp}=0\\
\frac{m}{2}\grad|H|^2+2\trace A_{\nabla^{\perp}_{\cdot}H}(\cdot)+2\trace(\bar R(\cdot,H)\cdot)^{\top}=0,
\end{cases}
\end{equation}
where $\Delta^{\perp}$ is the Laplacian in the normal bundle.
\end{theorem}

We end this section by recalling the following two results, which will be used later on.

\begin{lemma}[\cite{AdC,O}]\label{l:oku} Let $a_i$, $i=1,\ldots,m$, be real numbers such that $\sum_{i=1}^ma_i=0$ and $\sum_{i=1}^ma_i^2=b^2$, where $b=\cst\geq 0$. Then
$$
-\frac{m-2}{\sqrt{m(m-1)}}b^3\leq\sum_{i=1}^ma_i^3\leq\frac{m-2}{\sqrt{m(m-1)}}b^3,
$$
and equality holds in the right-hand $($left-hand$)$ side if and only if $(m-1)$ of the $a_i$'s are non-positive and equal $($$(m-1)$ of the $a_i$'s are non-negative and equal$)$.
\end{lemma}

\begin{theorem}[Omori-Yau Maximum Principle, \cite{Y}]\label{OY} If $\Sigma^m$ is a complete Riemannian
manifold with Ricci curvature bounded from below, then for any
smooth function $u\in C^2(\Sigma^m)$ with $\sup_{\Sigma^m}
u<+\infty$ there exists a sequence of points
$\{p_k\}_{k\in\mathbb{N}}\subset \Sigma^m$ satisfying
$$
\lim_{k\rightarrow\infty}u(p_k)=\sup_{\Sigma^m}u,\quad |\nabla u|(p_k)<\frac{1}{k}\quad\textnormal{and}\quad\Delta u(p_k)<\frac{1}{k}.
$$
\end{theorem}

\section{Parallel mean curvature surfaces in $\mathbb{C}P^n\times\mathbb{R}$ and $\mathbb{CH}^n\times\mathbb{R}$}

Let us denote by $N^n(c)\times\mathbb{R}$ the product spaces $\mathbb{C}P^n(c)\times\mathbb{R}$ or $\mathbb{CH}^n(c)\times\mathbb{R}$ as $c>0$ or $c<0$, respectively, endowed with their cosymplectic structure, and consider a PMC surface $\Sigma^2$ in $N^n(c)\times\mathbb{R}$.

In \cite[Theorem 3.2]{FR-TAMS} one proved that the $(2,0)$-part $Q^{(2,0)}$ of the quadratic form $Q$ defined on $\Sigma^2$ by
\begin{equation}\label{eq:Q}
Q(X,Y)=8|H|^2\langle\sigma(X,Y),H\rangle-c|H|^2\eta(X)\eta(Y)+3c\langle \varphi X,
H\rangle\langle \varphi Y, H\rangle,
\end{equation}
is holomorphic.

In the same way as in the case of PMC surfaces in complex space forms (see \cite{FP}), we obtain the following result.

\begin{theorem}\label{thm} Let $\Sigma^2$ be a complete non-minimal PMC surface with
non-negative Gaussian curvature $K$ in $N^n(c)\times\mathbb{R}$. Then either
\begin{enumerate}
\item the surface is flat; or

\item there exists a point $p\in\Sigma^2$ such that $K(p)>0$ and $Q^{(2,0)}$ vanishes on $\Sigma^2$.
\end{enumerate}
\end{theorem}

\begin{proof}
Let $M$ be an $m$-dimensional Riemannian manifold, and consider a Codazzi tensor, i.e., a symmetric operator $S$ on $M$, that satisfies the classical Codazzi equation $(\nabla_XS)Y=(\nabla_YS)X$, where $\nabla$ is the Levi-Civita connection on the manifold. We will need the following equation (equation $2.8$ in \cite{CY}), which generalizes J. Simons' result in \cite{JS}
\begin{equation}\label{delta}
\frac{1}{2}\Delta|S|^2=|\nabla S|^2+\sum_{i=1}^{m}\lambda_i(\trace S)_{ii}+\frac{1}{2}\sum_{i,j=1}^{m}R_{ijij}(\lambda_i-\lambda_j)^2,
\end{equation}
where $\lambda_i$, $1\leq i\leq m$, are the eigenvalues of $S$, and $R_{ijkl}$ are the components of the Riemannian curvature of $M$.

Now, let us consider an operator $S$, defined on $\Sigma^2$ by
\begin{eqnarray}\label{eq:S}
S&=&8|H|^2A_H+3c\langle \varphi H,\cdot\rangle tH-c|H|^2\eta(\cdot)T\\\nonumber&&-\left(8|H|^4+\frac{3c}{2}|tH|^2-\frac{c}{2}|H|^2|T|^2\right)\id.
\end{eqnarray}
We shall see in the following that $|S|^2$ is a bounded subharmonic function on the surface.

First, note that $S$ is symmetric and traceless, as it has the property
\begin{equation}\label{eq:SQ}
\langle SX,Y\rangle=Q(X,Y)-\frac{\trace Q}{2}\langle X,Y\rangle.
\end{equation}
Moreover, $Q^{(2,0)}$ vanishes on $\Sigma^2$ if and only if $S=0$ on the surface.

From \eqref{eq:SQ}, since $Q^{(2,0)}$ is holomorphic, it follows, as in \cite[Proposition 3.3]{B}, that $S$ is a Codazzi tensor. Therefore, from equation \eqref{delta} and $\trace S=0$, one obtains the following Simons type equation
\begin{equation}\label{eq:Simons}
\frac{1}{2}\Delta|S|^2=2K|S|^2+|\nabla S|^2.
\end{equation}

Next, consider the local orthonormal frame field $\{E_3=H/|H|,E_4,\ldots,E_{2n+1}\}$ in the normal bundle and $A_{\alpha}=A_{E_{\alpha}}$. We have $\trace A_3=2|H|$ and $\trace A_{\alpha}=0$, $\alpha>3$.

A straightforward computation, using the definition \eqref{eq:S} of $S$, leads to
\begin{eqnarray*}
\det A_3&=&|H|^2-\frac{1}{128|H|^6}|S|^2-\frac{9c^2}{256|H|^6}|tH|^2+\frac{3c}{64|H|^6}\langle S(tH),tH\rangle\\&&-\frac{c}{64|H|^4}\langle ST,T\rangle-\frac{c^2}{256|H|^2}|T|^4+\frac{3c^2}{128|H|^4}|T|^2|tH|^2.
\end{eqnarray*}

From \eqref{eq:R_product} one obtains the expression of the Gaussian curvature of $\Sigma^2$
\begin{equation}\label{K}
K=\frac{c}{4}(1+3\langle E_1,\varphi E_2\rangle-|T|^2)+\det A_3+\sum_{\alpha>3}\det A_{\alpha},
\end{equation}
where $\{E_1,E_2\}$ is a local orthonormal positively oriented frame field on the surface.

Since $\trace A_{\alpha}=0$, it follows that $\det A_{\alpha}\leq 0$, for all $\alpha>3$. Therefore, as $K\geq 0$, we have the inequality
$$
\frac{c}{4}(1+3\langle E_1,\varphi E_2\rangle^2-|T|^2)+\det A_3\geq 0.
$$
We also have $|tH|^2\leq|H|^2$, $c|\langle S(tH),tH\rangle|\leq(|c|/\sqrt{2})|tH||S|\leq(|c|/\sqrt{2})|H||S|$, and $c|\langle S(T),T\rangle|\leq(|c|/\sqrt{2})|T||S|\leq(|c|/\sqrt{2})|S|$.

In the following we shall prove that $|S|$ is bounded. We have two cases as $c<0$ or $c>0$.

If $c<0$, then $\det A_3\geq 0$ and therefore
$$
-\frac{1}{128|H|^4}|S|^2-\frac{4c}{64\sqrt{2}|H|^2}|S|+\frac{3c^2}{128}+|H|^4\geq 0,
$$
which leads to $|S|\leq(\sqrt{22\rho^2+256|H|^4}-4c)|H|^2/\sqrt{2}$.

If $c>0$, we have $\det A_3+c|H|^2\geq 0$, which implies
$$
-\frac{1}{128|H|^4}|S|^2+\frac{4c}{64\sqrt{2}|H|^2}|S|+\frac{3c^2}{128}+|H|^4+c|H|^2\geq 0,
$$
and then $|S|\leq(\sqrt{22c^2+256c|H|^2+256|H|^4}+4c)|H|^2/\sqrt{2}$.

The surface $\Sigma^2$ is complete and has non-negative Gaussian curvature, therefore it is a parabolic space (\cite{H}). Then, since $|S|^2$ is a bounded subharmonic function, from \eqref{eq:Simons}, it follows that $|S|$ is a constant and either $K=0$ on $\Sigma^2$ or there exists a point $p\in\Sigma^2$ such that $K(p)>0$ and $S=0$ on the surface.
\end{proof}

\begin{remark}\label{r_thm} Any surface satisfying the hypotheses of Theorem \ref{thm} also has the properties $|S|=\cst$ and $\nabla S=0$.
\end{remark}

Next, let us consider the case of PMC anti-invariant surfaces $\Sigma^2$. For such surfaces, we have one more holomorphic differential $Q'^{(2,0)}$, the traceless part of the quadratic form
$$
Q'(X,Y)=8\langle\sigma(X,Y),H\rangle-c\eta(X)\eta(Y),
$$
defined on $\Sigma^2$ (see \cite{FR-TAMS}).

Working exactly as in the general case, one can show that if a PMC anti-invariant surface has non-negative Gaussian curvature $K$, then either the surface is flat, or there exists a point $p$ on the surface such that $K(p)>0$ and, in this case, $Q'^{(2,0)}$ vanishes on $\Sigma^2$. Since also $Q^{(2,0)}$ vanishes on $\Sigma^2$, we use \cite[Corollaries 6.5, 6.10]{FR-TAMS} to conclude with the following theorem.

\begin{theorem}\label{thm_surfaces}
Let $\Sigma^2$ be a PMC anti-invariant surface in $N^n(c)\times\mathbb{R}$, $c\neq 0$, with non-negative Gaussian curvature. Then we have:
\begin{enumerate}

\item if $n\geq 2$, then $\Sigma^2$ is either flat, or pseudo-umbilical, or it lies in a product space $\bar N^4(c/4)\times\mathbb{R}$, where $\bar N^4(c/4)$ is a real space form with constant sectional curvature $c/4$ immersed as a totally geodesic totally real submanifold in $N^n(c)$;

\item if $n=2$, then $\Sigma^2$ is one of the embedded rotationally invariant CMC spheres $S_H^2\subset\bar N^2(c/4)\times\mathbb{R}$.

\end{enumerate}
\end{theorem}

\begin{remark} The geometry of the CMC spheres $S_H^2$ in $\bar N^2(c/4)\times\mathbb{R}$ was first studied in \cite{H-H,P-R}, before being thoroughly described in \cite{AR,AR2}.
\end{remark}

\section{A Simons type formula for PMC submanifolds in complex space forms}

As we have seen in the previous section, when dealing with surfaces in complex space forms or in product spaces, one has other, more suitable, options than working directly with the shape operator, which is not, in this context, a Codazzi tensor. However, when the dimension of the submanifold is higher, this seems to remain the best option in order to develop Simons type formulas. In complex space forms, such equations have been already obtained for PMC totally real surfaces (see \cite{Chen74}) and PMC Lagrangian submanifolds (see \cite{Chen77}), i.e., totally real submanifolds with maximum dimension. In the following we compute a similar formula in the general case of a PMC submanifold.

Let $\Sigma^m$, $m\leq 2n-1$, be an $m$-dimensional submanifold of a complex space form $N^n(c)$, with mean curvature vector field $H$. Assume that there exists a normal vector field $V$, such that $V$ is parallel in the normal bundle and $\trace A_V=\cst$. In this section, we shall compute the Laplacian of the squared norm of $A_V$, where $A$ is the shape operator of $\Sigma^m$ in $N^n(c)$.

First, from the Codazzi equation of the submanifold,
\begin{eqnarray*}
\langle \bar R(X,Y)Z,V\rangle&=&\langle\nabla^{\perp}_X\sigma(Y,Z),V\rangle-\langle\sigma(\nabla_XY,Z),V\rangle
-\langle\sigma(Y,\nabla_XZ),V\rangle\\&&-\langle\nabla^{\perp}_Y\sigma(X,Z),V\rangle+\langle\sigma(\nabla_YX,Z),V\rangle
+\langle\sigma(X,\nabla_YZ),V\rangle,
\end{eqnarray*}
where $\sigma$ is the second fundamental form of $\Sigma^m$, we have
\begin{equation*}
\langle \bar R(X,Y)Z,V\rangle=\langle(\nabla_XA_V)Y-(\nabla_YA_V)X,Z\rangle,
\end{equation*}
since $\nabla^{\perp}V=0$. Therefore, using \eqref{eq:curv_cpn}, we obtain
\begin{equation}\label{eq:Codazzi}
(\nabla_XA_V)Y=(\nabla_YA_V)X+\frac{c}{4}(\langle JX,V\rangle tY-\langle JY,V\rangle tX-2\langle X,JY\rangle tV)
\end{equation}
where $JX=tX+nX$, with $tX$ tangent to $\Sigma^m$ and $nX$ normal.

Next, the Laplacian of $|A_V|^2$ is given by the Weitzenb\"ock formula
\begin{equation}\label{eq:Laplacian}
\frac{1}{2}\Delta|A_V|^2=|\nabla A_V|^2+\langle\trace\nabla^2A_V,A_V\rangle,
\end{equation}
and the second term in the right hand side of this equation can
be computed by using a classical technique introduced in \cite[Section 1]{NS}.

Let us denote
$$
C(X,Y)=(\nabla^2 A_V)(X,Y)=\nabla_X(\nabla_YA_V)-\nabla_{\nabla_XY}A_V,
$$
and then the Ricci commutation formula reads as
\begin{equation}\label{eq:C} 
C(X,Y)=C(Y,X)+[R(X,Y),A_V].
\end{equation}
Now, consider an orthonormal basis $\{e_i\}_{i=1}^{m}$ in
$T_p\Sigma^m$, $p\in\Sigma^m$, and extend the $e_i$ to vector
fields $E_i$ in a neighborhood of $p$ such that $\{E_i\}$ is a
geodesic frame field around $p$, and denote $X=E_k$. We have
$$
(\trace\nabla^2A_V)X=\sum_{i=1}^mC(E_i,E_i)X.
$$

From equation \eqref{eq:Codazzi}, we have, at $p$, 
\begin{equation*}
C(E_i,X)E_i=\nabla_{E_i}((\nabla_{E_i}A_V)X)+\frac{c}{4}\nabla_{E_i}(\langle JX,V\rangle tY-\langle JY,V\rangle tX-2\langle X,JY\rangle tV)
\end{equation*}
and then, after a straightforward computation, using $\bar\nabla J=0$ and $\trace(tA_V)=0$, one obtains
\begin{eqnarray}\label{eq:1}
\sum_{i=1}^m C(E_i,X)E_i&=&\sum_{i=1}^m C(E_i,E_i)X-\frac{c}{4}\big(m\langle JX,V\rangle tH-t(A_{nV}X)\\\nonumber &&-t(A_V(tX))-m\langle JH,V\rangle tX+t(\sigma(X,tV))+A_{nX}(tV)\\\nonumber &&-2m\langle JH,X\rangle tV-2t(A_V(tX))+2A_{nV}(tX)\\\nonumber &&+\sum_{i=1}^m\left(\langle JX,V\rangle A_{nE_i}E_i-2\langle A_{nE_i}E_i,X\rangle tV\right)\big).
\end{eqnarray}

Still at $p$, $C(X,E_i)E_i=\nabla_{X}((\nabla_{E_i}A_V)E_i)$,
and then, from \eqref{eq:C}, it follows that
\begin{equation}\label{eq:new2}
C(E_i,X)E_i=\nabla_{X}((\nabla_{E_i}A_V)E_i)+[R(E_i,X),A_V]E_i.
\end{equation}

Since $\nabla_{E_i}A_V$ is symmetric, from \eqref{eq:Codazzi}, one obtains
\begin{eqnarray*}
\sum_{i=1}^m\langle(\nabla_{E_i}A_V)E_i,Z\rangle&=&\sum_{i=1}^m\langle E_i,(\nabla_{E_i}A_V)Z\rangle=\sum_{i=1}^m\langle E_i,(\nabla_{Z}A_V)E_i\rangle\\\nonumber &&+\frac{c}{4}\sum_{i=1}^m\langle E_i,\langle JE_i,V\rangle tZ-2\langle E_i,JZ\rangle tV\rangle\\\nonumber &=&Z(\trace A_V)+\frac{3c}{4}\langle t^2V,Z\rangle\\\nonumber &=&\frac{3c}{4}\langle t^2V,Z\rangle,
\end{eqnarray*}
for any vector $Z$ tangent to $\Sigma^m$, since $\trace A_V=\cst$. 

Next, using $\bar\nabla_X(J(tV))=J\bar\nabla_X(tV)$, this equation leads to
\begin{equation}\label{eq:cx}
\sum_{i=1}^m\nabla_{X}((\nabla_{E_i}A_V)E_i)=\frac{3c}{4}\left(tA_{nV}X-t^2A_VX+A_{n(tV)}X+t(\sigma(X,tV))\right).
\end{equation}

From the Gauss equation \eqref{eq:R_cpn}, after a straightforward
computation, one obtains
\begin{eqnarray}\label{R}
\sum_{i=1}^m[R(E_i,X),A_V]E_i&=&\frac{c}{4}(mA_VX+3t(A_V(tX))-3A_V(t^2X)\\\nonumber &&-(\trace A_V)X)\\\nonumber &&+\sum_{\alpha=m+1}^{2n}(A_{\alpha}A_VA_{\alpha}X-(\trace(A_VA_{\alpha})A_{\alpha}X\\\nonumber &&-A_VA_{\alpha}^2X+(\trace A_{\alpha})A_VA_{\alpha}X)
\end{eqnarray}

Finally, from \eqref{eq:1}, \eqref{eq:new2}, \eqref{eq:cx}, and \eqref{R}, we prove, also using \eqref{eq:Laplacian}, the following theorem.

\begin{theorem}\label{p:delta} Let $\Sigma^m$ be a submanifold of $N^n(c)$. If $V$ is a normal vector field parallel in the normal bundle, with $\trace A_V=\cst$, then
\begin{eqnarray}\label{eq:delta2}
\frac{1}{2}\Delta|A_V|^2&=&|\nabla A_V|^2+\frac{c}{4}\{m|A_V|^2+6|tA_V|^2+6\trace(A_V(tA_{nV}))\\\nonumber &&+3\trace(A_VA_{n(tV)})-(\trace A_V)^2+6\trace(A_V(t(A_Vt)))\\\nonumber &&+3\trace(\langle t(\sigma(A_V\cdot,tV)),\cdot\rangle)-3\trace(\langle t(\sigma(A_V(tV),\cdot),\cdot\rangle)\\\nonumber &&+3m\langle A_V(tH),tV\rangle\}+\sum_{\alpha=2m+1}^{2n}\{\trace(A_VA_{\alpha}A_VA_{\alpha}-A_V^2A_{\alpha}^2)\\\nonumber &&+\trace(A_{\alpha})\trace(A_V^2A_{\alpha})-(\trace(A_VA_{\alpha}))^2\},
\end{eqnarray}
where $\{E_{\alpha}\}_{\alpha=m+1}^{n+1}$ is a local orthonormal frame field in the normal bundle.
\end{theorem}

\begin{remark} As shown in \cite[Section 1]{NS}, if at a point $p\in\Sigma^m$ we consider a basis $\{X_i\}_{i=1}^{m}$ that diagonalizes $A_V$, one can prove that
$$
\sum_{i,j=1}^m\langle[R(X_i,X_j),A_V]X_i,A_VX_j\rangle=\sum_{i<j}(\lambda_i-\lambda_j)^2R_{ijij},
$$
where $\lambda_i$ are the eigenfunctions of $A_V$, and then equation \eqref{eq:delta2} can be alternatively written as
\begin{eqnarray}\label{eq:delta1}
\frac{1}{2}\Delta|A_V|^2&=&|\nabla A_V|^2+\frac{c}{4}\{3|tA_V|^2+6\trace(A_V(tA_{nV}))\\\nonumber &&+3\trace(A_VA_{n(tV)})+3\trace(A_V(t(A_Vt)))\\\nonumber &&+3\trace(\langle t(\sigma(A_V\cdot,tV)),\cdot\rangle)-3\trace(\langle t(\sigma(A_V(tV),\cdot),\cdot\rangle)\\\nonumber &&+3m\langle A_V(tH),tV\rangle\}+\sum_{i<j}(\lambda_i-\lambda_j)^2R_{ijij}.
\end{eqnarray}
\end{remark}

\section{Real hypersurfaces in complex space forms}

When dealing with hypersurfaces, the Simons type equations \eqref{eq:delta2} and \eqref{eq:delta1} simplify. In this section we consider real hypersurfaces $\Sigma^{2n-1}$ in a complex space form $N^n(c)$. Such a hypersurface has a unit normal vector field $\nu$ that is parallel in the normal bundle. 

\begin{proposition}\label{hyper_complex} If $\Sigma^{2n-1}$ is a CMC real hypersurface in a complex space form $N^n(c)$, then
\begin{eqnarray}\label{eq:hyper2}
\frac{1}{2}\Delta|A|^2&=&|\nabla A|^2-|A|^4+(2n-1)|H|\trace A^3+\frac{c}{4}\{(2n-4)|A|^2\\\nonumber &&+3|tA-At|^2+3(2n-1)|H|\langle A(J\nu),J\nu\rangle-(2n-1)^2|H|^2\}.
\end{eqnarray}

\end{proposition}

\begin{proof} The vector field $E_1=-J\nu$ is tangent and therefore we can consider an orthonormal frame field $\{E_1,E_2,\ldots,E_{n-1},E_n=JE_2,\ldots,E_{2n-1}=JE_{n-1}\}$ on $\Sigma^{2n-1}$.

It is then easy to see that, for any $i\neq 1$, we have 
$$
\langle t(\sigma(AE_i,J\nu),E_i\rangle=\langle t(\sigma(E_i,A(J\nu)),E_i\rangle=0.
$$
Moreover, we also have 
$$
\langle t(\sigma(AE_1,J\nu),E_1\rangle=\langle t(\sigma(E_1,A(J\nu)),E_1\rangle=\langle\sigma(AE_1,E_1),\nu\rangle=|AE_1|^2
$$
and $|tA|^2=|At|^2=|A|^2-|AE_1|^2$, which implies that 
$$
|tA|^2+\trace(A(t(At)))=(1/2)|tA-At|^2.
$$ 
Since $n\nu=0$ and $n(t\nu)=n(J\nu)=-\nu$, we conclude by using equation \eqref{eq:delta2}.
\end{proof} 

\begin{remark} In the case of CMC real hypersurfaces, equation \eqref{eq:delta1} becomes
\begin{eqnarray}\label{eq:hyper1}
\frac{1}{2}\Delta|A|^2&=&|\nabla A|^2+\frac{c}{8}\{-6|A|^2+3|tA-At|^2+6(2n-1)|H|\langle A(J\nu),J\nu\rangle\}\\\nonumber &&+\sum_{i<j}(\lambda_i-\lambda_j)^2R_{ijij}.
\end{eqnarray}
\end{remark}

\begin{remark} Homogeneous real hypersurfaces in $\mathbb{C}P^n(c)$ were completely classified in \cite{T} and they are divided into five types $A-E$, with type $A$ hyperurfaces further subdivided as $A_1$ and $A_2$. In $\mathbb{C}H^n(c)$, these hypersurfaces were classified in \cite{BT} (see also \cite{NR} for further details on real hypersurfaces in complex space forms). 
\end{remark}

\begin{remark} We recall that the Ricci tensor $S$ of a real hypersurface is given by (see, for example, \cite{NR})
\begin{equation*}
SX=\langle SX,X\rangle=\frac{(2n+1)c}{4}X-\frac{3c}{4}\langle X,J\nu\rangle J\nu+(2n-1)|H|AX-A^2X
\end{equation*}
and then, for a unit vector field $X$ tangent to the hypersurface, we have
\begin{equation}\label{Ricci}
\ric X=\frac{(2n+1)c}{4}-\frac{3c}{4}\langle X,J\nu\rangle^2+(2n-1)|H|\langle AX,X\rangle-|AX|^2.
\end{equation}

\end{remark}

Henceforth, for the sake of simplicity, we will consider only the cases when $c=4$ or $c=-4$ and will denote the ambient spaces by $\mathbb{C}P^n$ or $\mathbb{C}H^n$, respectively.

\begin{proposition} There are no complete minimal real hypersurfaces in $\mathbb{C}P^n$ such that $|A|^2\leq k_0<2(n-1)$, where $k_0\geq 0$ is a real constant.
\end{proposition}

\begin{proof} Let us consider a complete minimal real hypersurface $\Sigma^{2n-1}$ in $\mathbb{C}P^n$ and assume that there exists a non-negative constant $k_0<2(n-1)$ such that $|A|^2\leq k_0$. 

Since $\Sigma$ is minimal, equation \eqref{eq:hyper2} becomes
$$
\frac{1}{2}\Delta |A|^2=|\nabla A|^2-|A|^4+2(n-2)|A|^2+3|tA-At|^2,
$$
and then, as we know that $|\nabla A|^2\geq 4(n-1)$ (\cite[Theorem 1.11]{NR}), one obtains
\begin{eqnarray}\label{hyper-minimal}
\frac{1}{2}\Delta |A|^2&\geq&-|A|^4+2(n-2)|A|^2+4(n-1)=(|A|^2+2)(2(n-1)-|A|^2)\\\nonumber&\geq& (|A|^2+2)(2(n-1)-k_0)>0.
\end{eqnarray}

From \eqref{Ricci}, one can see that, for a unit vector field $X$ tangent to $\Sigma$, we have
$$
\ric X\geq\frac{2n+3}{4}-k_0,
$$
which means that the Ricci curvature of the hypersurface is bounded from below. We can then apply the Omori-Yau Maximum Principle (Theorem \ref{OY}) to the function $u=|A|^2$. It follows that there exists a sequence of points $\{p_k\}_{k\in\mathbb{N}}\subset\Sigma$ satisfying
$$
\lim_{k\rightarrow\infty}|A|^2(p_k)=\sup_{\Sigma}|A|^2\quad\textnormal{and}\quad\Delta |A|^2(p_k)<\frac{1}{k},
$$
and, therefore, equation \eqref{hyper-minimal} gives
$$
0\geq(\sup_{\Sigma}|A|^2+2)(2(n-1)-k_0)>0,
$$
a contradiction which concludes the proof.
\end{proof}

\begin{remark} In \cite[Theorem 2]{L} it is treated the case of compact real hypersurfaces $\Sigma^{2n-1}$ with $|A|^2\leq 2(n-1)$, and it is proved that they are hypersurfaces of type $A$ and $|A|^2=2(n-1)$. This result can be also recovered by using Proposition \ref{hyper_complex}.
\end{remark}

\begin{theorem}\label{thm_cmc} Let $\Sigma^{2n-1}$ be a complete CMC non-minimal real hypersurface in $\mathbb{C}P^n$ with sectional curvature $K\geq 1$ and
$$
|A|^2\geq\frac{(2n-1)^2|H|^2-4(n-1)}{2n-4}.
$$
Then $\Sigma^{2n-1}$ is a geodesic sphere of radius $r=\pi/4$ with $|H|=2(n-1)/(2n-1)$.
\end{theorem}

\begin{proof} Let $\lambda_i$ denote the eigenfunctions of $A$. Since $K\geq 1$, we have
$$
\sum_{i<j}(\lambda_i-\lambda_j)^2R_{ijij}\geq\sum_{i<j}(\lambda_i-\lambda_j)^2=(2n-1)|A|^2-(2n-1)^2|H|^2,
$$
and $\ric J\nu\geq 2n-2$. We also have, from equation \eqref{Ricci},
$$
\ric J\nu=2n-2+(2n-1)|H|\langle A(J\nu),J\nu\rangle-|A(J\nu)|^2,
$$
which implies $(2n-1)|H|\langle A(J\nu),J\nu\rangle\geq|A(J\nu)|^2\geq 0$.

Next, from \eqref{eq:hyper1} and again using $|\nabla A|^2\geq 4(n-1)$, one obtains
\begin{eqnarray}\label{eq:cmc}
\frac{1}{2}\Delta|A|^2&=&|\nabla A|^2-3|A|^2+\frac{3}{2}|tA-At|^2+3(2n-1)|H|\langle A(J\nu),J\nu\rangle\\\nonumber &&+\sum_{i<j}(\lambda_i-\lambda_j)^2R_{ijij}\\\nonumber&\geq&(2n-4)|A|^2-(2n-1)^2|H|^2+4(n-1)\geq 0.
\end{eqnarray}
Consider $u=(2n-4)|A|^2-(2n-1)^2|H|^2+4(n-1)\geq 0$ and note that, since the Ricci curvature of $\Sigma$ is bounded from below, the Omori-Yau Maximum Principle holds on the hypersurface. For $u$, we have that there exists a sequence of points $\{p_k\}_{k\in\mathbb{N}}\subset\Sigma$ satisfying
$$
\lim_{k\rightarrow\infty}u(p_k)=\sup_{\Sigma}u\quad\textnormal{and}\quad\Delta u(p_k)<\frac{1}{k}.
$$
It follows, from \eqref{eq:cmc}, that $\sup_{\Sigma}u=0$ and therefore $u$ vanishes on $\Sigma$ and all the inequalities we used to get \eqref{eq:cmc} become equalities. 

Hence, we have $K=1$, $tA=At$, $|\nabla A|^2=4(n-1)$, $A(J\nu)=0$, and
$$
|A|^2\geq\frac{(2n-1)^2|H|^2-4(n-1)}{2n-4}.
$$
From $tA=At$ (and also $|\nabla A|^2=4(n-1)$) it follows that $\Sigma$ is a homogeneous type~$A$ hypersurface (\cite[Theorem 4.1]{NR}) and, moreover, a geodesic sphere (which is a type $A1$ hypersurface) with $|A|^2=2n-2$, which allows us to compute the exact value of the mean curvature (\cite[Theorems 3.13, 3.14]{NR}).
\end{proof}

The last two results of this section are concerned with compact real hypersurfaces in $\mathbb{C}H^n$.

\begin{proposition}\label{min_chn} There are no minimal compact real hypersurfaces in $\mathbb{C}H^n$ such that
\begin{equation}\label{ineq}
|A|^2\leq\sqrt{n^2+12n+12}-n-4.
\end{equation}
\end{proposition}

\begin{proof} Let $\Sigma^{2n-1}$ be a compact minimal real hypersurface in $\mathbb{C}P^n$ and assume that it satisfies the inequality \eqref{ineq}. Then equation \eqref{eq:hyper2} gives
\begin{eqnarray*}
\frac{1}{2}\Delta |A|^2&=&|\nabla A|^2-|A|^4-2(n-2)|A|^2-3|tA-At|^2\\&\geq&-|A|^4-(2n+8)|A|^2+4(n-1)\geq 0,
\end{eqnarray*}
where we have again used the facts that $|\nabla A|^2\geq 4(n-1)$ and also $|tA-At|^2\leq 4|tA|^2\leq 4|A|^2$. By integrating this inequality over $\Sigma$, one sees that
$$
|A|^2=\sqrt{n^2+12n+12}-n-4
$$
and that all the above inequalities must be equalities. But this means that $|\nabla A|^2=4(n-1)$, which is equivalent to $\Sigma$ being a type $A$ hypersurface (\cite[Corollary 4.4]{NR}) and then $tA=At$ (\cite[Theorem 4.1]{NR}). Since $|tA-At|^2=4|A|^2$, all these lead to $|A|=0$, which is a contradiction.
\end{proof}

\begin{proposition} If $\Sigma^{2n-1}$ is a compact CMC real hypersurface in $\mathbb{C}H^n$ with sectional curvature $K\geq -1$, then there exists a point $p\in\Sigma^{2n-1}$ such that, at $p$,
$$
|A|^2>C(n,|H|)=\frac{\sqrt{(8n+17)(2n-1)^2|H|^2+32(n^2-1)}-3(2n-1)|H|}{4(n+1)}.
$$
\end{proposition}

\begin{proof} As in the proof of Theorem \ref{thm_cmc}, if $\lambda_i$ are the eigenfunctions of $A$, from $K\geq -1$, one obtains
$$
\sum_{i<j}(\lambda_i-\lambda_j)^2R_{ijij}\geq-\sum_{i<j}(\lambda_i-\lambda_j)^2=-(2n-1)|A|^2+(2n-1)^2|H|^2.
$$

Next, from \eqref{eq:hyper1}, we have
\begin{eqnarray*}
\frac{1}{2}\Delta|A|^2&=&|\nabla A|^2-\frac{3}{2}|tA-At|^2+3|A|^2-3(2n-1)|H|\langle A(J\nu),J\nu\rangle\\&&+\sum_{i<j}(\lambda_i-\lambda_j)^2R_{ijij}\\&\geq&-(2n+2)|A|^2-3(2n-1)|H||A|+(2n-1)^2|H|^2+4(n-1),
\end{eqnarray*}
since $|\nabla A|^2\geq 4(n-1)$, $|tA-At|^2\leq 4|tA|^2\leq 4|A|^2$, and $\langle A(J\nu),J\nu\rangle\leq |A|$. 

Now, if we assume that $|A|\leq C(n,|H|)$, then we come to a contradiction in the same way as in the proof of Proposition \ref{min_chn}. Therefore, $|A|>C(n,|H|)$ holds at a point $p$ on the hypersurface.
\end{proof}

\begin{remark} Detailed information on the admissible range for the sectional curvature of real hypersurfaces in complex space forms can be found in \cite{MTT}.
\end{remark}

\section{Totally real submanifolds in complex space forms}

In this section we consider PMC totally real submanifolds $\Sigma^m$, $m\leq n$, in non-flat complex space forms $N^n(c)$ and study them under one of the additional geometric hypotheses that the vector field $JH$ is either normal or tangent to $\Sigma^m$.

Let us first consider the case when $JH$ is normal to the submanifold. The following result is a direct consequence of Theorem \ref{p:delta} and equation \eqref{eq:delta1}.

\begin{proposition}\label{p:tot1} If $\Sigma^m$, $m\leq n$, is a PMC totally real submanifold of a complex space form $N^n(c)$ with $JH$ normal to $\Sigma^m$, then
\begin{eqnarray}\label{eq:delta_tot_2}
\frac{1}{2}\Delta|A_H|^2&=&|\nabla A_H|^2+\frac{c}{4}\{m|A_H|^2-m^2|H|^4\}+m\trace A_H^3\\\nonumber&&-\sum_{\alpha=2m+1}^{2n}(\trace(A_HA_{\alpha}))^2,
\end{eqnarray}
where $\{E_{\alpha}\}_{\alpha=m+1}^{n+1}$ is a local orthonormal frame field in the normal bundle, or, equivalently,
\begin{equation}\label{eq:delta_tot_1}
\frac{1}{2}\Delta|A_H|^2=|\nabla A_H|^2+\sum_{i<j}(\lambda_i-\lambda_j)^2R_{ijij},
\end{equation}
where $\lambda_i$ are the eigenfunctions of $A_H$.
\end{proposition}

\begin{proof} Let $\{E_1,\ldots,E_m\}$ be an orthonormal frame field on $\Sigma$. Then, since $\Sigma$ is totally real and $JH$ is normal, we can consider the orthonormal frame field 
$$
\left\{E_{m+1}=\frac{H}{|H|},E_{m+2}=\frac{JH}{|H|},E_{m+3}=JE_1,\ldots,E_{2m+2}=JE_m,E_{2m+3},\ldots,E_{2n}\right\}
$$ 
in the normal bundle. 

From the Ricci equation of the submanifold,
\begin{equation}\label{eq:ricci}
\langle R^{\perp}(X,Y)V,U\rangle=\langle[A_V,A_U]X,Y\rangle+\langle\bar R(X,Y)V,U\rangle,
\end{equation}
which holds for all normal vector fields $U$ and $V$, and the expression \eqref{eq:curv_cpn} of the curvature tensor of $N^n(c)$, as $\nabla^{\perp}H=0$ and $\Sigma$ is totally real, we have $[A_H,A_{E_{\alpha}}]=0$, for all $\alpha\geq m+1$. Since $tA_H=0$, $nH=JH$, $tH=0$, and $ntH=0$, we get \eqref{eq:delta_tot_2}. Equation \eqref{eq:delta_tot_1} follows immediately from \eqref{eq:delta1}.
\end{proof}

\begin{remark} If $\Sigma^m$ is compact and has positive sectional curvature, it is easy to see, from \eqref{eq:delta_tot_1}, that it is also pseudo-umbilical with $\nabla A_H=0$.
\end{remark}

In \cite{FLMO} it is proved that the mean curvature $|H|$ of a proper-biharmonic submanifold $\Sigma^m$ of $\mathbb{C}P^n=\mathbb{C}P^n(4)$ with $JH$ normal, satisfies $|H|\leq 1$ and the equality occurs if and only if $\Sigma^m$ is pseudo-umbilical and PMC. Our next result shows that if the submanifold is assumed to be PMC and also totally real we have a gap in the admissible range of $|H|$. 

\begin{proposition}\label{prop:pb_cpn} If $\Sigma^m$ is a proper-biharmonic PMC totally real submanifold of $\mathbb{C}P^n$ with $JH$ normal and $|H|\geq(m-2)/m$, then $\nabla A_H=0$ and either
\begin{enumerate}

\item $|H|=1$, and $\Sigma^m$ is pseudo-umbilical; or

\item $|H|=(m-2)/m$, and $\Sigma^m$ is locally a product.
\end{enumerate}
\end{proposition}

\begin{proof} From Theorem \ref{thm:split} and formula \eqref{eq:curv_cpn}, one can see that $\Sigma^m$ satisfies the hypotheses if and only if
\begin{equation*}
\trace\sigma(\cdot,A_H\cdot)=mH,
\end{equation*} 
which is equivalent to $|A_H|^2=m|H|^2$ and $\trace(A_HA_U)=0$, for all normal vector fields $U$ orthogonal to $H$.

Next, from Proposition \ref{p:tot1}, one obtains
\begin{equation}\label{delta_normal}
0=\frac{1}{2}\Delta|A_H|^2=|\nabla A_H|^2+m\trace A_H^3-m^2|H|^4.
\end{equation}

Let us now consider $\phi_H=A_H-|H|^2\id$, the traceless part of $A_H$. We have
\begin{equation*}
\trace A_H^3=\trace\phi_H^3+3|H|^2|\phi_H|^2+m|H|^6
\end{equation*}
and
\begin{equation*}
|\phi_H|^2=|A_H|^2-m|H|^4=m|H|^2(1-|H|^2).
\end{equation*}
Then, from \eqref{delta_normal} and Lemma \ref{l:oku}, it follows
\begin{eqnarray*}
0=\frac{1}{2}\Delta|\phi_H|^2&\geq&-\frac{m(m-2)}{\sqrt{m(m-1)}}|\phi_H|^3+2m|H|^2|\phi_H|^2\\&=&m|\phi_H|^2\left(2|H|^2-\frac{m-2}{\sqrt{m(m-1)}}|\phi_H|\right)\geq 0,
\end{eqnarray*}
as $|H|\geq(m-2)/m$.

In conclusion, either $\phi_H=0$, which means that $\Sigma$ is pseudo-umbilical and $|H|=1$, or $|\phi_H|=(2(m-2)\sqrt{m-1})/(m\sqrt{m})$, which is equivalent to $|H|=(m-2)/m$. In both cases $\nabla A_H=0$ and, moreover, all the inequalities we used become equalities.

If $|H|=(m-2)/m$, we can proceed exactly as in the similar case of PMC proper-biharmonic submanifolds in the Euclidean sphere (\cite[Theorem~3.11]{BO}). Thus, from Lemma \ref{l:oku}, it follows that the principal curvatures in the direction of $H$ are constant functions on $\Sigma$, given by
$$
\lambda_1=\ldots=\lambda_{m-1}=\lambda=\frac{m-2}{m},\quad\lambda_m=\mu=-\frac{m-2}{m}.
$$
which define the distributions
$$
T_{\lambda}=\{X\in T\Sigma: A_HX = \lambda X\}\quad\textnormal{and}\quad T_{\mu}=\{X\in T\Sigma: A_HX=\mu X\}.
$$
Since $A_H$ is parallel, these distributions are mutually orthogonal, smooth, involutive and parallel, which means that we can apply the de
Rham Decomposition Theorem (see \cite{dR}).

Hence, for every point $p\in\Sigma$ there exists a neighborhood $U\subset\Sigma$ which is isometric to a product $\Sigma^{m-1}\times (-\epsilon,\epsilon)$, where
$\Sigma^{m-1}$ is an integral submanifold for $T_{\lambda}$ through $p$ and the interval $I=(-\epsilon,\epsilon)$ corresponds to the integral curves of a unit vector field $Y\in T_{\mu}$ on U. We note that $\Sigma^{m-1}$ is a totally geodesic submanifold in $U$ and the integral curves of $Y$ are geodesics in $U$. Moreover, $Y$ is a parallel vector field on $U$.
\end{proof}

Next, we consider the case when the vector field $JH$ is tangent to the submanifold $\Sigma^m$ and obtain the following proposition.

\begin{proposition}\label{p:tot2} If $\Sigma^m$, $m\leq n$, is a complete PMC totally real submanifold of a complex space form $N^n(c)$ with $JH$ tangent to $\Sigma^m$, then
\begin{eqnarray}\label{eq:delta_tot_2_2}
\frac{1}{2}\Delta|A_H|^2&=&|\nabla A_H|^2+\frac{c}{4}\{(m-3)|A_H|^2+3m\langle A_H(JH),JH\rangle\\\nonumber&&+4(m-1)|H|^2-m^2|H|^4\}\\\nonumber&&+m\trace A_H^3-\sum_{\alpha=2m+1}^{2n}(\trace(A_HA_{\alpha}))^2,
\end{eqnarray}
where $\{E_{\alpha}\}_{\alpha=m+1}^{n+1}$ is a local orthonormal frame field in the normal bundle, or, equivalently,
\begin{eqnarray}\label{eq:delta_tot_1_2}
\frac{1}{2}\Delta|A_H|^2&=&|\nabla A_H|^2+\frac{c}{4}\{-3|A_H|^2+3m\langle A_H(JH),JH\rangle+3(m-1)|H|^2\}\\\nonumber&&+\sum_{i<j}(\lambda_i-\lambda_j)^2R_{ijij},
\end{eqnarray}
where $\lambda_i$ are the eigenfunctions of $A_H$.
\end{proposition}

\begin{proof} Let $\{E_1=-JH/|H|,E_2,\ldots,E_m\}$ be an orthonormal frame field on $\Sigma$ and 
$$
\{E_{m+1}=H,E_{m+2}=JE_2,\ldots,E_{2m}=JE_m,E_{2m+1},\ldots,E_{2n}\}
$$ 
be the corresponding orthonormal frame field in the normal bundle.

First, using the Ricci equation \eqref{eq:ricci} of $\Sigma$, the fact that $\nabla^{\perp}H=0$, and \eqref{eq:curv_cpn}, we have
\begin{eqnarray}\label{eq:sigma}
\trace(\langle t\sigma(A_H\cdot,JH),\cdot\rangle)-\langle t\sigma(A_H(JH),\cdot),\cdot\rangle)=\\\nonumber=\sum_{i=1}^m(\langle\sigma(A_HJH,E_i),JE_i\rangle-\langle\sigma(A_HE_i,JH),JE_i\rangle)=\sum_{i=1}^m\langle[A_{JE_i},A_H]JH,E_i\rangle&\\\nonumber=-\sum_{i=1}^m\langle\bar R(JH,E_i)JE_i,H\rangle=-\sum_{i=1}^m(-|H|^2+\langle JH,E_i\rangle^2)=(m-1)|H|^2.
\end{eqnarray}
Replacing in \eqref{eq:delta1}, and taking into account that $tA_H=0$, $nH=0$, and $n(tH)=-H$, one obtains \eqref{eq:delta_tot_1_2}.

Now, we have, again from the Ricci equation of $\Sigma$ and \eqref{eq:curv_cpn}, that $[A_H,A_{E_{\alpha}}]=0$, for any $\alpha\geq 2m+1$, and
\begin{equation*}
\langle[A_H,A_{JE_i}]X,Y\rangle=-\langle\bar R(X,Y)H,JE_i\rangle=\langle JH,X\rangle\langle Y,E_i\rangle-\langle JH,Y\rangle\langle X,E_i\rangle,
\end{equation*}
for any $i\geq 2$. It follows that
\begin{eqnarray*}
\sum_{i\geq 2}\sum_{j=1}{m}\langle A_H([A_{JE_i},A_H])A_{JE_i}E_j,E_j\rangle=\sum_{i,j}\langle[A_{JE_i},A_H]A_{JE_i}E_j,A_HE_j\rangle\\=\sum_{i\geq 2}(\langle A_{JE_i}E_i, A_HJH\rangle-\langle A_{JE_i}JH,A_HE_i\rangle)\\=\trace(\langle t\sigma(A_H\cdot,JH),\cdot\rangle-\langle t\sigma(A_H(JH),\cdot),\cdot\rangle))=(m-1)|H|^2.
\end{eqnarray*}
We come to equation \eqref{eq:delta_tot_2_2} by replacing this term in \eqref{eq:delta1} and again using $tA_H=0$, $nH=0$, and $n(tH)=-H$.
\end{proof}

\begin{proposition} Let $\Sigma^m$, $m\leq n$, be a compact PMC totally real submanifold in a non-flat complex space form $N^n(c)$ with $JH$ tangent to $\Sigma^m$, positive sectional curvature, and $c\langle\phi_H(JH),JH\rangle\geq 0$, where $\phi_H=A_H-|H|^2\id$ is the traceless part of $A_H$. If 
$$
c|A_H|^2\leq c(m-1)|H|^2,
$$
then $\Sigma^m$ is pseudo-umbilical and $|H|^2=(m-1)/m$.
\end{proposition}

\begin{proof} From equation \eqref{eq:delta_tot_1_2}, we have
$$
\frac{1}{2}\Delta|A_H|^2\geq\frac{3c}{4}\{-|A_H|^2+(m-1)|H|^2\}\geq 0,
$$
and then, integrating over $\Sigma$, one obtains
$$
0\geq c\int_{\Sigma}\{-|A_H|^2+(m-1)|H|^2\}\geq 0,
$$
which means that $|A_H|^2=(m-1)|H|^2$, and all the inequalities we used become equalities. Since the sectional curvature is positive, it follows that $\lambda_i=\lambda_j$, for all $i,j\in\{1,2,\ldots,m\}$, i.e., $\Sigma$ is pseudo-umbilical. This means that $|A_H|^2=m|H|^4$, and, therefore, $|H|^2=(m-1)/m$.
\end{proof}

\section{A Simons type formula for PMC submanifolds in $\mathbb{C}P^n\times\mathbb{R}$ and $\mathbb{C}H^n\times\mathbb{R}$ and applications}

Let $\Sigma^m$, $m\leq 2n$, be a submanifold of a product space $N^n(c)\times\mathbb{R}$, where $N^n(c)$ is a complex space form, such that there exists a normal vector field $V$ satisfying $\nabla^{\perp}V=0$ and $\trace A_V=\cst$. Consider the cosymplectic structure $(\varphi,\eta,\xi)$ on $N^n(c)\times\mathbb{R}$ and, as in previous sections, decompose the vector fields $\varphi X=tX+nX$ and $\xi=T+N$ in their tangent and normal parts, respectively. Then, working exactly as in the case of complex space forms and using the fact that the cosymplectic structure is parallel, one can derive the following Simons type equation. 

\begin{theorem}\label{thm:delta2_prod}
Let $\Sigma^m$ be a submanifold of $N^n(c)\times\mathbb{R}$. If $V$ is a normal vector field parallel in the normal bundle, with $\trace A_V=\cst$, then
\begin{eqnarray}\label{eq:delta2_prod}
\frac{1}{2}\Delta|A_V|^2&=&|\nabla A_V|^2+\frac{c}{4}\{m|A_V|^2-2m|A_VT|^2+6|tA_V|^2-|T|^2|A_V|^2\\\nonumber &&+3(\trace A_V)\langle A_VT,T\rangle+6\trace(A_V(tA_{nV}))\\\nonumber &&+3\trace(A_VA_{n(tV)})-(\trace A_V)^2+6\trace(A_V(t(A_Vt)))\\\nonumber &&+3\trace(\langle t(\sigma(A_V\cdot,tV)),\cdot\rangle)-3\trace(\langle t(\sigma(A_V(tV),\cdot),\cdot\rangle)\\\nonumber &&+m\langle V,N\rangle(\trace(A_VA_N)-(\trace A_V)\langle H,N\rangle)\\\nonumber&&+3m\langle A_V(tH),tV\rangle\}+\sum_{\alpha=2m+1}^{2n}\{\trace(A_VA_{\alpha}A_VA_{\alpha}-A_V^2A_{\alpha}^2)\\\nonumber &&+\trace(A_{\alpha})\trace(A_V^2A_{\alpha})-(\trace(A_VA_{\alpha}))^2\},
\end{eqnarray}
where $\{E_{\alpha}\}_{\alpha=m+1}^{n+1}$ is a local orthonormal frame field in the normal bundle.
\end{theorem}

\begin{remark} As in the case of complex space forms, for a submanifold $\Sigma^m$ as in Theorem \ref{thm:delta2_prod}, equation \eqref{eq:delta2_prod} can be written as
\begin{eqnarray}\label{eq:delta1_prod}
\frac{1}{2}\Delta|A_V|^2&=&|\nabla A_V|^2+\frac{c}{4}\{3|tA_V|^2+(\trace A_V)\langle A_VT,T\rangle\\\nonumber &&-m|A_VT|^2+6\trace(A_V(tA_{nV}))\\\nonumber&&+3\trace(A_VA_{n(tV)})+3\trace(A_V(t(A_Vt)))\\\nonumber &&+3\trace(\langle t(\sigma(A_V\cdot,tV)),\cdot\rangle)-3\trace(\langle t(\sigma(A_V(tV),\cdot),\cdot\rangle)\\\nonumber &&+m\langle V,N\rangle(\trace(A_VA_N)-(\trace A_V)\langle H,N\rangle)\\\nonumber &&+3m\langle A_V(tH),tV\rangle\}+\sum_{i<j}(\lambda_i-\lambda_j)^2R_{ijij}.
\end{eqnarray}
where $\lambda_i$ are the eigenfunctions of $A_V$.
\end{remark}

In the following, we consider hypersurfaces $\Sigma^{2n}$ in $N^n(c)\times\mathbb{R}$ and denote by $\nu$ a normal unit vector field.

\begin{proposition}\label{hyper} If $\Sigma^{2n}$ is a CMC hypersurface in $N^n(c)\times\mathbb{R}$, then
\begin{eqnarray}\label{eq:hyper2_prod}
\frac{1}{2}\Delta|A|^2&=&|\nabla A|^2-|A|^4+2n|H|\trace A^3+\frac{c}{4}\{(4n-(2n+4)|T|^2)|A|^2\\\nonumber &&-4n|AT|^2+3|tA-At|^2+6n|H|\langle AT,T\rangle\\\nonumber&&+6n|H|\langle A(\varphi\nu),\varphi\nu\rangle-4n^2(2-|T|^2)|H|^2\}
\end{eqnarray}

\end{proposition}

\begin{proof} If at a point $q\in\Sigma$ we have $T_q=0$, then one can see that \eqref{eq:hyper2_prod} actually turns out to be \eqref{eq:delta2_prod}, at $q$. Let us assume that $T\neq 0$  at a point $p$ on the surface. Then, at $p$, the vector fields $\varphi\nu$ and $T$ are tangent and orthogonal to each other. Moreover, $|\varphi\nu|^2=1-\eta^2(\nu)=|T|^2$. Thus we can define $E_1=-(\varphi\nu)/|T|$ and $E_2=T/|T|$, two unit tangent vectors. Now, if $E_3$ is another unit tangent vector orthogonal to $E_1$ and $E_2$, it follows that
$$
\langle\varphi E_3,\nu\rangle=-\langle E_3,\varphi\nu\rangle=0\quad\textnormal{and}\quad \langle\varphi E_3,\varphi E_3\rangle=\langle E_3,E_3\rangle,
$$
since $E_3\perp\xi$, which means that $\varphi E_3$ is also a unit tangent vector, and then
$$
\langle\varphi E_3,E_1\rangle=-\frac{1}{|T|}(\langle E_3,\nu\rangle-|N|\langle E_3,\xi\rangle)=0,\quad\langle\varphi E_3,E_2\rangle=\frac{1}{|T|}\langle\varphi E_3,\xi\rangle=0.
$$ 
Therefore we can consider an orthonormal basis
$$
\{E_1,E_2,E_3,E_4=\varphi E_3\ldots,E_{2n-1},E_{2n}=\varphi E_{2n-1}\},
$$
in $T_p\Sigma$. It is also easy to compute 
$$
\varphi E_1=-|N|E_2+|T|\nu\quad\textnormal{and}\quad\varphi E_2=|N|E_1.
$$

Next, for any $i\geq 2$, we have $\langle t(\sigma(AE_i,\varphi\nu),E_i\rangle=\langle t(\sigma(E_i,A(\varphi\nu)),E_i\rangle=0$. Moreover, we also obtain 
$$
\langle t(\sigma(AE_1,\varphi\nu),E_1\rangle=\langle t(\sigma(E_1,A(\varphi\nu)),E_1\rangle=\langle\sigma(AE_1,E_1),\nu\rangle=|A(\varphi\nu)|^2
$$
and $|tA|^2=|At|^2=|A|^2-|AE_1|^2-|AE_2|^2$, which implies that 
$$
|tA|^2+\trace(A(t(At)))=(1/2)|tA-At|^2.
$$
Since $n\nu=0$ and $n(t\nu)=n(\varphi\nu)=-|T|^2\nu$, we replace in equation \eqref{eq:delta2_prod} and conclude the proof.
\end{proof} 

\begin{remark} For CMC hypersurfaces equation \eqref{eq:delta1_prod} reads as
\begin{eqnarray}\label{eq:hyper1_prod}
\frac{1}{2}\Delta|A|^2&=&|\nabla A|^2+\frac{c}{4}\{(2n-(2n+3)|T|^2)|A|^2-2n|AT|^2+\frac{3}{2}|tA-At|^2\\\nonumber&&+2n|H|(\langle A(\varphi\nu),\varphi\nu\rangle+\langle AT,T\rangle)-4n^2(1-|T|^2)|H|^2\}\\\nonumber &&+\sum_{i<j}(\lambda_i-\lambda_j)^2R_{ijij}.
\end{eqnarray}
\end{remark}

As before, again for the sake of simplicity, we will only consider the case $c=4$.

\begin{proposition} Let $\Sigma^{2n}$ be a complete minimal hypersurface in $\mathbb{C}P^n\times\mathbb{R}$ such that
$$
\sup_{\Sigma}\{|A|^2-(4n-(6n+4)|T|^2\}\leq k_0<0,
$$
where $k_0$ is a negative real constant. Then $\Sigma^m$ is a slice $\mathbb{C}P^n\times\{t_0\}$.
\end{proposition}

\begin{proof} For a minimal hypersurface, equation \eqref{eq:hyper2_prod} becomes
$$
\frac{1}{2}|A|^2=|\nabla A|^2-|A|^4+(4n-(2n+4)|T|^2)|A|^2-4n|AT|^2+3|tA-At|^2.
$$
Since $|AT|^2\leq|T|^2|A|^2$ and $|tA-At|^2\geq 0$, one obtains
\begin{equation}\label{eq:h}
\frac{1}{2}|A|^2\geq |A|^2(4n-(6n+4)|T|^2-|A|^2)\geq 0.
\end{equation}
The Ricci curvature of our hypersurface can be computed by using equation \eqref{eq:R_product}. Thus, for a unit tangent vector field $X$, we have
$$
\ric X=2n-1-(2n-2)\langle X,T\rangle^2+3|tX|^2-|T|^2-|AX|^2,
$$
which shows that the Ricci curvature is bounded from below and, therefore, the Omori-Yau Maximum Principle holds on $\Sigma$. We apply it to the function $u=|A|^2$ and obtain, from \eqref{eq:h}, that $\sup_{\Sigma}|A|^2=0$, which means that $\Sigma$ is totally geodesic, and hence a slice.
\end{proof}

\begin{proposition} Let $\Sigma^{2n}$ be a compact CMC hypersurface in $\mathbb{C}P^n\times\mathbb{R}$ with sectional curvature $K\geq -1$ such that
$$
|A|\leq C(n,|H|)=\frac{4n|H|(2\sqrt{n+1}-1)}{4n+3}.
$$
Then $\Sigma^m$ is a slice $\mathbb{C}P^n\times\{t_0\}$.
\end{proposition} 

\begin{proof} The condition on the sectional curvature implies
$$
\sum_{i<j}(\lambda_i-\lambda_j)^2R_{ijij}\geq-\sum_{i<j}(\lambda_i-\lambda_j)^2=4n^2|H|^2-2n|A|^2,
$$
and then, from \eqref{eq:hyper1_prod} and since $|A|\leq C(n,|H|)$, we have
$$
\frac{1}{2}\Delta|A|^2\geq(-(4n+3)|A|^2+4n|H||A|+4n^2|H|^2)|T|^2\geq 0,
$$
where we have used $|AT|^\leq|T|^2|A|^2$, $\langle A(\varphi\nu),\varphi\nu\rangle\geq-|T|^2|A|$, $\langle AT,T\rangle\geq-|T|^2|A|$, and $|tA-At|^2\geq 0$.

Next, we integrate over $\Sigma$ and obtain
$$
0\geq(-(4n+3)|A|^2+4n|H||A|+4n^2|H|^2)|T|^2\geq 0,
$$
which means that either $|A|=C(n,|H|)$ or $T=0$ and all inequalities we have used become equalities. It follows that $|AT|^2=|T|^2|A|^2$, $\langle AT,T\rangle=-|T|^2|A|$, $\langle A(\varphi\nu,\varphi\nu\rangle=-|T|^2|A|^2$, and $tA=At$. As it is easy to see, if $T\neq 0$ at a point $p\in\Sigma$, we have $A_p=0$ and, therefore, $|H|_p=0$, which shows that $|H|=0$ everywhere, and then, from our hypotheses, $A=0$ everywhere, i.e., $\Sigma$ is a totally geodesic hypersurface and thus a slice $\mathbb{C}P^n\times\{t_0\}$, for some $t_0\in\mathbb{R}$.
\end{proof}

\begin{proposition} Let $\Sigma^m$ be a proper-biharmonic PMC anti-invariant submanifold of $\mathbb{C}P^n\times\mathbb{R}$ with $\varphi H$ normal, such that 
$$
|H|^2>\frac{(m-1)(m^2+4)+(m-2)\sqrt{(m-1)(m-2)(m^2+m+2)}}{2m^3}
$$
and the norm of its second fundamental form $\sigma$ bounded. Then $\Sigma^m$ lies in $\mathbb{C}P^n$ as a totally real pseudo-umbilical submanifold and $|H|=1$.
\end{proposition}

\begin{proof} From Theorem \ref{thm:split} and formula \eqref{eq:curv_product}, it follows that $\Sigma$ satisfies
$$
\trace\sigma(\cdot,A_H\cdot)=(m-|T|^2)H-m\langle H,N\rangle N\quad\textnormal{and}\quad\langle H,N\rangle T=0.
$$
At any point of $\Sigma$ we have $\langle H,N\rangle=0$ or $T=0$. Assume that $T=0$ at a point $p\in\Sigma$ and then $T=0$ in a neighborhood $V_p$ of $p$. On $V_p$ we have $\langle X,\xi\rangle=0$ for all tangent vector fields $X$. Therefore $0=\langle\bar\nabla_YX,\xi\rangle=\langle\sigma(X,Y),\xi\rangle=0$, which implies that $\langle H,N\rangle=0$ on $V_p$. Thus $H$ is orthogonal to $\xi$, and also
$$
|A_H|^2=(m-|T|^2)|H|^2\quad\textnormal{and}\quad\trace(A_HA_U)=0,
$$
for all normal vector fields $U$ orthogonal to $H$.

From Theorem \ref{thm:delta2_prod}, one obtains
$$
\frac{1}{2}\Delta|A_H|^2=|\nabla A_H|^2-m^2|H|^4+m\trace A_H^3.
$$

Working in the exact same way as in the proof of Theorem~4.9 in \cite{FOR}, we can prove that $\Sigma$ lies in $\mathbb{C}P^n$ and, since $|H|>(m-2)/m$, we conclude by using Theorem \ref{prop:pb_cpn}.
\end{proof}

\end{document}